\documentclass[12pt]{amsart}

\usepackage{amsmath,amsfonts,amssymb,amsthm,mathabx,shuffle,latexsym,mathtools,stmaryrd}
\usepackage{enumitem}
\usepackage[usenames, dvipsnames, svgnames]{xcolor}
\usepackage{ytableau}
\usepackage{tikz-cd}
\usepackage{pbox}
\newtheorem{theorem}{Theorem}[section]
\newtheorem{lemma}[theorem]{Lemma}
\newtheorem{proposition}[theorem]{Proposition}
\newtheorem{corollary}[theorem]{Corollary}

\theoremstyle{definition}

\newtheorem{example}[theorem]{Example}

\theoremstyle{remark}
\newtheorem{remark}[theorem]{Remark}

\numberwithin{equation}{section}

\newcommand{\Des}{\ensuremath{\mathrm{Des}}}

\newcommand{\Sym}{\ensuremath{\mathrm{Sym}}}
\newcommand{\QSym}{\ensuremath{\mathrm{QSym}}}

\newcommand{\SYRT}{\ensuremath{\mathrm{SYRT}}}

\newcommand{\R}{\ensuremath{\mathcal{R}}}

\newcommand{\comp}{\ensuremath{\mathrm{comp}}}

\newcommand{\C}{\ensuremath{\mathbb{C}}}

\newcommand{\Set}{\ensuremath{\mathbb{S}}}

\newcommand{\excise}[1]{}

\newlength\cellsize 

\savebox2{%
\begin{picture}(18,18)
\put(0,0){\line(1,0){18}}
\put(0,0){\line(0,1){18}}
\put(18,0){\line(0,1){18}}
\put(0,18){\line(1,0){18}}
\end{picture}}

\newcommand\cellify[1]{\def\thearg{#1}\def\nothing{}%
\ifx\thearg\nothing\vrule width0pt height\cellsize depth0pt%
  \else\hbox to 0pt{\usebox2\hss}\fi%
  \vbox to 18\unitlength{\vss\hbox to 18\unitlength{\hss$#1$\hss}\vss}}

\newcommand\tableau[1]{\vtop{\let\\=\cr
\setlength\baselineskip{-12000pt}
\setlength\lineskiplimit{12000pt}
\setlength\lineskip{0pt}
\halign{&\cellify{##}\cr#1\crcr}}}

\usepackage[colorinlistoftodos]{todonotes}

\ytableausetup{centertableaux}

\begin{document}


\title{$0$-Hecke modules for Young row-strict quasisymmetric Schur functions}  

\author[J. Bardwell]{Joshua Bardwell}
\address{Department of Mathematics and Statistics, University of Otago, 730 Cumberland St., Dunedin 9016, New Zealand}
\email{barjo848@student.otago.ac.nz}

\author[D. Searles]{Dominic Searles}
\address{Department of Mathematics and Statistics, University of Otago, 730 Cumberland St., Dunedin 9016, New Zealand}
\email{dominic.searles@otago.ac.nz}


\date{December 22, 2020}



\begin{abstract}
We construct modules of the $0$-Hecke algebra whose images under the quasisymmetric characteristic map are the Young row-strict quasisymmetric Schur functions. This provides a representation-theoretic interpretation of this basis of quasisymmetric functions, answering a question of Mason and Niese (2015). Additionally, we classify when these modules are indecomposable. 
\end{abstract}

\maketitle

%
\section{Introduction}
%
\label{sec:introduction}

The \emph{Schur functions} form a basis of the algebra of symmetric functions $\Sym$ that plays an important role in myriad areas of mathematics. Schur functions arise, for example, as representatives of Schubert classes in the cohomology of Grassmannian varieties, as characters of irreducible polynomial representations of general linear groups, and as images of irreducible characters of symmetric groups under the characteristic map. 
The algebra $\Sym$ is a subalgebra of the algebra $\QSym$ of quasisymmetric functions, and it is therefore natural to seek bases of $\QSym$ that reflect or extend properties of the Schur functions. 
Examples of such bases of $\QSym$ include the \emph{fundamental quasisymmetric functions} \cite{Gessel}, the \emph{dual immaculate functions} \cite{BBSSZ:2}, the \emph{quasisymmetric Schur functions} \cite{HLMvW11:QS}, the \emph{row-strict quasisymmetric Schur functions} \cite{Mason.Remmel}, and the \emph{extended Schur functions} \cite{Assaf.Searles:3}. 

Quasisymmetric functions have a representation-theoretic interpretation in terms of \emph{$0$-Hecke algebras}, which are a certain deformation of the group algebra of symmetric groups. There is an isomorphism of algebras between the Grothendieck group of $0$-Hecke representations and $\QSym$, known as the \emph{quasisymmetric characteristic} \cite{DKLT}. 
Analogously to the role Schur functions play for irreducible representations of symmetric groups, the fundamental quasisymmetric functions are the images of the irreducible representations of $0$-Hecke algebras under the quasisymmetric characteristic map \cite{DKLT}. 
This raises the question of interpreting other bases of $\QSym$ as quasisymmetric characteristics of certain families of $0$-Hecke modules. Indeed all the aforementioned bases, save the row-strict quasisymmetric Schur functions, have been interpreted in this way; such modules were constructed for dual immaculate functions in \cite{BBSSZ}, for quasisymmetric Schur functions in \cite{TvW:1}, and for extended Schur functions in \cite{Sea20}.  

The row-strict quasisymmetric Schur functions are conjugate to the well-studied quasisymmetric Schur functions under an extension of the famous $\omega$ involution from $\Sym$ to $\QSym$ \cite{Malvenuto.Reutenauer}. The Schur functions expand positively in the row-strict quasisymmetric Schur basis via an elegant formula \cite{Mason.Remmel}, as they do into the quasisymmetric Schur basis \cite{HLMvW11:QS}. In \cite{Mason.Niese}, a closely-related variant called the \emph{Young row-strict quasisymmetric Schur functions} was introduced and many properties of this basis discovered, including an analogue of the Littlewood-Richardson rule. The question of interpreting the Young row-strict quasisymmetric Schur functions in terms of $0$-Hecke modules was raised in \cite{Mason.Niese}. 

Additionally, there has been recent interest in further understanding the structure of $0$-Hecke modules that arise in this context, particularly regarding indecomposability. All $0$-Hecke modules for dual immaculate quasisymmetric functions and extended Schur functions are indecomposable (\cite{BBSSZ}, \cite{Sea20} respectively). The same is not true for the modules for quasisymmetric Schur functions, however \cite{TvW:1} provided a direct-sum formula for these modules and used this to classify which modules are indecomposable. Later, \cite{Koenig} showed that all components of this direct sum decomposition are indecomposable. 
In \cite{TvW:2}, $0$-Hecke modules for a generalization of quasisymmetric Schur functions were constructed, and \cite{Choi.Kim.Nam.Oh20a} established several structural results concerning these modules, including classifying indecomposability. Moreover, \cite{Choi.Kim.Nam.Oh20b} determined the projective covers for the modules for dual immaculate quasisymmetric functions and extended Schur functions, and for those modules in \cite{TvW:2} that are indecomposable.

In this paper, we answer the question of Mason and Niese \cite{Mason.Niese} by constructing $0$-Hecke modules whose quasisymmetric characteristics are the Young row-strict quasisymmetric Schur functions. Moreover, we classify when these modules are indecomposable. Proving this classification turns out to be more involved than the analogous arguments for indecomposability of modules for dual immaculate, extended Schur and quasisymmetric Schur functions. The condition that classifies indecomposability for modules for Young row-strict quasisymmetric Schur functions in fact agrees with that for modules for quasisymmetric Schur functions \cite{TvW:1}, at least up to a reversal of the compositions indexing the functions; this is due to similarity in the definitions of the standard tableaux defining each of these bases. However, significant difference in the \emph{descent} structure between these families of tableaux leads to different module structure, necessitating a different approach to the proof. 

The paper is organized as follows. In Section 2 we review the necessary background concerning the fundamental and row-strict Young quasisymmetric Schur bases of $\QSym$ and the $0$-Hecke algebra. In Section 3 we define a $0$-Hecke action on the \emph{standard Young row-strict tableaux} of \cite{Mason.Niese} and prove that the quasisymmetric characteristics of the corresponding $0$-Hecke modules are precisely the Young row-strict quasisymmetric Schur functions. We give a formula for a decomposition of these modules into a direct sum of nonzero submodules, and prove that each of these submodules is generated by a single tableau. We also establish precisely when this direct sum formula has only a single summand. In Section 4 we prove that a certain submodule is always indecomposable. The indecomposability classification then follows from the fact that when the direct sum formula has only one summand, this submodule is the entire module.

\section{Background}

\subsection{Quasisymmetric functions}

A \emph{composition} is a finite sequence $\alpha=(\alpha_1, \ldots , \alpha_k)$ of positive integers. The \emph{parts} of $\alpha$ are the integers $\alpha_i$ for $1\le i \le k$, and the number $k$ of parts is the \emph{length} of $\alpha$, denoted $\ell(\alpha)$.  When the parts of $\alpha$ sum to $n$, we write $|\alpha|=n$ and say that $\alpha$ is a \emph{composition of $n$}, denoted $\alpha\vDash n$. We also write ${\rm max}(\alpha)$ to denote ${\rm max}\{\alpha_1, \ldots , \alpha_{\ell(\alpha)}\}$. 

Given a composition $\alpha=(\alpha_1, \ldots , \alpha_k)$ of $n$, define a subset $\mathbb{S}(\alpha)$ of $\{1, \ldots , n-1\}$ by $\mathbb{S}(\alpha) = \{\alpha_1, \alpha_1+\alpha_2, \ldots , \alpha_1+\alpha_2 +\cdots + \alpha_{k-1}\}$. The map $\alpha\mapsto \mathbb{S}(\alpha)$ is a bijection between compositions of $n$ and subsets of $\{1,\ldots , n-1\}$. Its inverse $\comp_n$ is defined by $\comp_n(\{x_1 < \cdots < x_r\}) = (x_1,x_2-x_1, \ldots , x_r - x_{r-1}, n-x_r)$. For example, $\mathbb{S}(3,2,2) = \{3,5\}\subseteq \{1, \ldots , 6\}$ and $\comp_7(\{1,3,4\}) = (1,2,1,3)\vDash 7$.

Denote by $\C[[x_1,x_2, \ldots]]$ the algebra of formal power series of bounded degree in infinitely many commuting variables. The algebra $\QSym$ of \emph{quasisymmetric functions} is a subalgebra of $\C[[x_1,x_2, \ldots]]$, and bases of $\QSym$ are indexed by compositions. The \emph{monomial quasisymmetric functions} $\{M_\alpha\}$ \cite{Gessel}, defined by
\[M_\alpha = \sum_{i_1< i_2 < \cdots < i_k} x_{i_1}^{\alpha_1}  \cdots x_{i_{k}}^{\alpha_{k}}\]
form a basis of $\QSym$. Another important basis is the \emph{fundamental quasisymmetric functions} $\{F_\alpha\}$ \cite{Gessel}, defined by
\[F_\alpha = \sum_{\beta \text{ refines }\alpha} M_\beta,\]
where $\beta$ \emph{refines} $\alpha$ if $\alpha$ can be obtained by summing consecutive parts of $\beta$.

\begin{example}
Let $\alpha = (1,2,2)$. We have
\[M_{(1,2,2)} = \sum_{i<j<k}x_ix_j^2x_k^2\]
and
\[F_{(1,2,2)} = M_{(1,2,2)} + M_{(1,1,1,2)} +  M_{(1,2,1,1)} + M_{(1,1,1,1,1)}.\]
\end{example}

We now introduce a third basis: the Young row-strict quasisymmetric Schur functions, which are defined in terms of certain tableaux of composition shape. 
The \emph{diagram} $D(\alpha)$ of a composition $\alpha$ is the array of cells having $\alpha_i$ left-justified cells in row $i$. We use French notation for composition diagrams, i.e., the rows are numbered from bottom to top. Let $(c,r)$ denote the cell in row $r$ and column $c$. We say the cell $(c+1,r)\in D(\alpha)$ is \emph{right-adjacent} to the cell $(c,r)$, and that $(c,r)$ is \emph{left-adjacent} to $(c+1,r)$. 

\begin{example}
Let $\alpha=(3,2,2)$. Then $D(\alpha) = \tableau{ {\ } & {\ }  \\ {\ }  & {\ } \\ {\ } & {\ } & {\ } }$.
\end{example}

Let $\alpha\vDash n$. A \emph{standard Young row-strict composition tableau} \cite{Mason.Niese} of shape $\alpha$ is a bijective assignment $T$ of the cells of $D(\alpha)$ to entries $1, \ldots , n$ satisfying the following conditions:

\begin{enumerate}
\item[(R1)] Entries increase from left to right along rows
\item[(R2)] Entries increase from bottom to top in the first column
\item[(R3)] If cells $(c,r)$ and $(c+1,r')$ for $r'<r$ are in $D(\alpha)$ and $T(c,r)<T(c+1,r')$, then $T(c+1,r)<T(c+1,r')$, where $T(c+1,r)$ is defined to be $\infty$ if $(c+1,r)\notin D(\alpha)$.
\end{enumerate}

Pictorially, (R3) states that for any three cells arranged as below (which we refer to as a \emph{triple}), if $a<c$ then $b<c$.

  \begin{displaymath}
    \begin{array}{l}
        \tableau{ a & b } \\[-0.5\cellsize]  \hspace{2\cellsize} \vdots  \hspace{0.4\cellsize} \\ \tableau{  & c } 
    \end{array}
  \end{displaymath}
  
We denote the set of all standard Young row-strict composition tableaux of shape $\alpha$ by $\SYRT(\alpha)$. For $T\in \SYRT(\alpha)$, when a cell with entry $j$ is right-adjacent to a cell with entry $i$ in $T$, we say $j$ is right-adjacent to $i$ (and $i$ is left-adjacent to $j$).

\begin{remark}
The term \emph{row-strict} comes from a semistandard version of these tableaux, in which an entry $i$ may appear more than once. For semistandard Young row-strict tableaux, entries may be repeated in columns but are required to strictly increase along rows; see \cite{Mason.Niese}. We will not need the semistandard version in this paper.
\end{remark}

Define the \emph{descent set} $\Des(T)$ of $T\in \SYRT(\alpha)$ to be the entries $i$ such that $i+1$ is strictly to the right of $i$ in $T$.

\begin{example}\label{ex:SYRT}
Let $\alpha = (3,2,2)\vDash 7$. The tableaux in $\SYRT(\alpha)$, along with their descent sets, are shown below.
  \begin{displaymath}
   \begin{array}{c@{\hskip1.5\cellsize}c@{\hskip1.5\cellsize}c@{\hskip1.5\cellsize}c@{\hskip1.5\cellsize}c}
\tableau{ 6 & 7 \\ 4 & 5 \\ 1 & 2 & 3 } & \tableau{ 6 & 7 \\ 3 & 5 \\ 1 & 2 & 4 }  & \tableau{ 5 & 6 \\ 4 & 7 \\ 1 & 2 & 3 }  &  \tableau{ 4 & 6 \\ 3 & 7 \\ 1 & 2 & 5 }  & \tableau{ 5 & 6 \\ 3 & 7 \\ 1 & 2 & 4 } \\ \\
\{1,2,4,6\} & \{1,3,6\} & \{1,2,5\} & \{1,4\} & \{1,3,5\} 
\end{array}
  \end{displaymath}
\end{example}

For $\alpha\vDash n$, the Young row-strict quasisymmetric Schur function $\R_\alpha$ \cite{Mason.Niese} is defined by
\[\R_\alpha = \sum_{T\in \SYRT(\alpha)} F_{\comp_n(\Des(T))}.\]

\begin{example}
By Example~\ref{ex:SYRT}, we have 
\[\R_{(3,2,2)} = F_{(1,1,2,2,1)} + F_{(1,2,3,1)}+F_{(1,1,3,2)}+F_{(1,3,3)} + F_{(1,2,2,2)}.\]
\end{example}

\subsection{$0$-Hecke algebras}

The $0$-Hecke algebra $H_n(0)$ is the $\C$-algebra with generators $T_1, \ldots , T_{n-1}$ subject to the relations
\begin{align*}
T_i^2 & =  T_i  \quad \mbox{ for all } 1 \le i \le n-1 \\
T_iT_j & =  T_jT_i  \quad \mbox{ for all } i, j  \mbox{ such that } |i-j|\ge 2 \\
T_iT_{i+1}T_i & = T_{i+1}T_iT_{i+1}  \quad  \mbox{ for all } 1\le i \le n-2.
\end{align*}

Given a representation $X$ of $H_n(0)$, let $[X]$ denote its isomorphism class. The \emph{Grothendieck group} $\mathcal{G}_0(H_n(0))$ is the linear span of the isomorphism classes of the finite-dimensional representations of $H_n(0)$, modulo the relation $[Y]=[X]+[Z]$ for each short exact sequence $0\rightarrow X\rightarrow Y\rightarrow Z\rightarrow 0$ of $H_n(0)$-representations $X,Y,Z$. Define
\[\mathcal{G} = \bigoplus_{n\ge 0} \mathcal{G}_0(H_n(0)).\]
There are $2^{n-1}$ irreducible representations of $H_n(0)$, all of which are one-dimensional \cite{Norton}. They may be indexed by the $2^{n-1}$ compositions of $n$; let $\mathbf{F}_\alpha$ denote the irreducible representation corresponding to the composition $\alpha$. Let $\{v_\alpha\}$ be a basis of $\mathbf{F}_\alpha$. The following action of the $T_i$ on $v_\alpha$ gives the structure of $\mathbf{F}_\alpha$ as a $H_n(0)$-representation.
\begin{align}\label{eq:irreps}
T_i(v_\alpha) = \begin{cases} v_\alpha & \mbox{ if } i\notin \Set(\alpha) \\
					       0 & \mbox{ if } i \in \Set(\alpha).
					       \end{cases}
 \end{align}

The set $\{[\mathbf{F}_\alpha]\}$ as $\alpha$ ranges over all compositions forms a basis of $\mathcal{G}$. 
There is an algebra isomorphism $ch:\mathcal{G}\rightarrow \QSym$ \cite{DKLT} defined by setting $ch([\mathbf{F}_\alpha]) = F_\alpha$. For any $H_n(0)$-module $X$, the \emph{quasisymmetric characteristic} of $X$ is the quasisymmetric function $ch([X])$.

\section{Modules for Young row-strict quasisymmetric Schur functions}\label{sec:modules}

In this section we construct $H_n(0)$-modules $\bf{R}_\alpha$ whose quasisymmetric characteristics are the Young row-strict quasisymmetric Schur functions $\mathcal{R}_\alpha$, answering a question of Mason and Niese \cite{Mason.Niese}. We then show that each ${\bf R}_\alpha$ decomposes as a direct sum of nonzero submodules, each of which is generated by a single $\SYRT$, and characterize the compositions for which this direct sum has only one summand. These structural results will be needed for the indecomposability classification in Section 4.

\subsection{$0$-Hecke modules on standard Young row-strict tableaux}

Let $\alpha$ be a composition of $n$. For each $1\le i \le n-1$ and each $T\in \SYRT(\alpha)$, define 
\[\pi_i(T) = \begin{cases} T & \mbox{ if } $i+1$ \mbox{ is weakly left of } $i$ \mbox{ in } $T$ \\
                                        0 & \mbox{ if } $i+1$ \mbox{ is right-adjacent to } $i$ \mbox{ in } $T$ \\
                                        s_i(T) & \mbox{ otherwise}
                                        \end{cases}
                                        \]
                                        where $s_i(T)$ is the filling of $D(\alpha)$ obtained by swapping the entries $i$ and $i+1$ in $T$.

\begin{example}
Let $\alpha = (3,2,2)$, and let 
\[T =  \tableau{ 5 & 6 \\ 3 & 7 \\ 1 & 2 & 4 }   \in \SYRT(\alpha).\]
Then $\pi_2(T)=\pi_4(T)=\pi_6(T)=T$, $\pi_1(T) = \pi_5(T) = 0$ and 
\[\pi_3(T) = s_3(T) =  \tableau{ 5 & 6 \\ 4 & 7 \\ 1 & 2 & 3 }   \in \SYRT(\alpha).\]
\end{example}

Let ${\bf R}_\alpha$ denote the $\C$-vector space spanned by $\SYRT(\alpha)$. 

\begin{lemma}\label{lem:imagecontained}
Let $T\in \SYRT(\alpha)$. Then for any $1\le i \le n-1$ we have $\pi_i(T) \in {\bf R}_\alpha$.
\end{lemma}
\begin{proof}
This is immediate in the case where $\pi_i(T)=T$ or $\pi_i(T)=0$. Suppose that $\pi_i(T)=s_i(T)$; we need to show $s_i(T)\in \SYRT(\alpha)$. Observe that $i$ and $i+1$ must be in different rows in $T$, since if they were in the same row then $i+1$ would be right-adjacent to $i$ by (R1). Therefore, swapping $i$ and $i+1$ does not change the relative order of entries in either the row containing $i$ or the row containing $i+1$, so $s_i(T)$ satifies (R1). The entries $i$ and $i+1$ are necessarily in different columns of $T$, therefore swapping $i$ and $i+1$ does not change the relative order of entries in any column, so $s_i(T)$ satisfies (R2). For (R3), observe that for any triple in $T$ that involves at most one of $i$ and $i+1$, the relative order of entries in that triple is the same in $s_i(T)$ as it is in $T$. Therefore, we only need to consider the case that $i$ and $i+1$ belong to the same triple in $T$. But such a triple cannot exist in $T$: since $i+1$ is strictly right of $i$ in a different row, this would mean $i$ occupies the top-left cell of the triple and $i+1$ the bottom-right cell. Then the entry right-adjacent to $i$ would be larger than $i+1$, meaning (R3) is not satisfied for $T$.
\end{proof}

\begin{theorem}\label{thm:0Hecke}
The operators $\pi_i$ define an $H_n(0)$-action on ${\bf R}_\alpha$. 
\end{theorem}
\begin{proof}
Let $T\in \SYRT(\alpha)$. Lemma~\ref{lem:imagecontained} establishes that for any $i$, $\pi_i(T)\in {\bf R}_\alpha$. We need to confirm that the operators $\pi_i$ satisfy the relations for the generators of the $0$-Hecke algebra.

If $i+1$ is weakly left of $i$ in $T$, then $\pi_i(T)=T$, so $\pi_i^2(T) = T = \pi_i(T)$. If $i+1$ is right-adjacent to $i$ in $T$, then $\pi_i(T)=0$, so $\pi_i^2(T) = 0 = \pi_i(T)$. Otherwise $\pi_i(T)=s_i(T)$, in which case $\pi_i^2(T) = \pi_i(s_i(T)) = s_i(T) = \pi_i(T)$, where the middle equality follows from the fact that $i+1$ is left of $i$ in $s_i(T)$. Hence $\pi_i^2=\pi_i$ for all $1\le i \le n-1$.

If $|i-j|\ge 2$, then there is no overlap between $\{i, i+1\}$ and $\{j,j+1\}$. In other words, $\pi_i$ and $\pi_j$ apply to disjoint pairs of entries in $T$ and therefore it is immediate that $\pi_i\pi_j = \pi_j\pi_i$.

We now show that $\pi_i\pi_{i+1}\pi_i(T) = \pi_{i+1}\pi_i\pi_{i+1}(T)$  via the following cases.

\begin{enumerate}
\item $i+1$ is right-adjacent to $i$
\item $i+1$ is weakly left of $i$
 \begin{enumerate}
 \item $i+2$ is right-adjacent to $i+1$
 \item $i+2$ is weakly left of $i+1$
 \item $i+2$ is strictly right of $i+1$ and not right-adjacent to $i+1$
 \end{enumerate}
\item $i+1$ is strictly right of $i$ and not right-adjacent to $i$
 \begin{enumerate}
 \item $i+2$ is right-adjacent to $i+1$
 \item $i+2$ is weakly left of $i+1$
 \item $i+2$ is strictly right of $i+1$ and not right-adjacent to $i+1$
 \end{enumerate}
\end{enumerate}

\noindent
(1): Here $\pi_{i}(T)=0$, so $\pi_i\pi_{i+1}\pi_{i}(T) = 0$. If $\pi_{i+1}(T)=0$ we are done. If $\pi_{i+1}(T) = T$, then $\pi_{i+1}\pi_{i}\pi_{i+1}(T) =  \pi_{i+1}\pi_{i}(T) = 0$. If $\pi_{i+1}(T) = s_{i+1}(T)$, then $i+2$ is strictly right of $i+1$ (and also $i$) and in a different row to $i+1$ (and $i$). Hence $\pi_i\pi_{i+1}(T) = s_is_{i+1}(T)$, and the cells occupied by $i+1$ and $i+2$ in $s_is_{i+1}(T)$ are exactly the cells occupied by $i$ and $i+1$ respectively in $T$. Since these cells are adjacent, we have  $\pi_{i+1}\pi_{i}\pi_{i+1}(T) = \pi_{i+1}s_is_{i+1}(T) = 0$.

\noindent
(2)(a): Here $\pi_i(T) = T$ and $\pi_{i+1}(T)=0$, so $\pi_i\pi_{i+1}\pi_i(T)  = \pi_{i+1}\pi_i\pi_{i+1}(T)=0$.

\noindent
(2)(b): Here $\pi_i(T)=T$ and $\pi_{i+1}(T)=T$, so $\pi_i\pi_{i+1}\pi_i(T)=\pi_{i+1}\pi_i\pi_{i+1}(T)=T$.

\noindent
(2)(c): Here $\pi_i(T)=T$ and $\pi_{i+1}(T)=s_{i+1}(T)$. If $i+2$ is right-adjacent to $i$ in $T$, then $i+1$ is right-adjacent to $i$ in $s_{i+1}(T)$, so $\pi_i\pi_{i+1}\pi_{i}(T) = \pi_i (s_{i+1}(T)) = 0$ and $\pi_{i+1}\pi_{i}\pi_{i+1}(T) = \pi_{i+1}\pi_i(s_{i+1}(T))=0$. If $i+2$ is weakly left of $i$ in $T$, then $i+1$ is weakly left of $i$ in $s_{i+1}(T)$, so $\pi_i(s_{i+1}(T)) = s_{i+1}(T)$. Then $\pi_i\pi_{i+1}\pi_{i}(T) = \pi_i (s_{i+1}(T)) = s_{i+1}(T)$ and $\pi_{i+1}\pi_{i}\pi_{i+1}(T) =  \pi_{i+1}\pi_i(s_{i+1}(T)) = \pi_{i+1} (s_{i+1}(T)) = s_{i+1}(T)$. Finally, if $i+2$ is strictly right of $i$ in $T$ and not right-adjacent to $i$, then $i+1$ is strictly right of $i$ in $s_{i+1}(T)$ (and not right-adjacent to $i$), so $\pi_i(s_{i+1}(T)) = s_is_{i+1}(T)$. Then $\pi_i\pi_{i+1}\pi_{i}(T) = \pi_i (s_{i+1}(T)) = s_is_{i+1}(T)$ and $\pi_{i+1}\pi_{i}\pi_{i+1}(T) =  \pi_{i+1}\pi_i(s_{i+1}(T)) = \pi_{i+1} (s_is_{i+1}(T)) = s_is_{i+1}(T)$, where the last equality is due to the fact that the cells occupied by $i+1$ and $i+2$ in $s_is_{i+1}(T)$ are those occupied by $i$ and $i+1$ in $T$.

\noindent
(3)(a): Here $\pi_{i+1}(T)=0$, so $\pi_{i+1}\pi_i\pi_{i+1}(T)=0$. We have $\pi_i(T) = s_i(T)$, and in $s_i(T)$,  $i+1$ is strictly left of $i+2$ and not left-adjacent to $i+2$, so $\pi_i\pi_{i+1}\pi_i(T) = \pi_i\pi_{i+1}s_i(T) = \pi_i(s_{i+1}s_i(T))$. In $s_{i+1}s_i(T)$, the cells occupied by $i$ and $i+1$ are exactly those occupied by $i+1$ and $i+2$ respectively in $T$, which are adjacent. Hence $\pi_i(s_{i+1}s_i(T))=0$.

\noindent
(3)(b): Here $\pi_i(T) = s_i(T)$ and $\pi_{i+1}(T)=T$. If $i+2$ is right-adjacent to $i$ in $T$, then $i+2$ is right-adjacent to $i+1$ in $s_i(T)$, so $\pi_i\pi_{i+1}\pi_{i}(T) = \pi_i \pi_{i+1} (s_{i}(T)) = 0$ and $\pi_{i+1}\pi_{i}\pi_{i+1}(T) = \pi_{i+1}\pi_i(T) = \pi_{i+1}(s_i(T)=0$. If $i+2$ is weakly left of $i$ in $T$, then $i+1$ is weakly left of $i+2$ in $s_i(T)$, so $\pi_{i+1}(s_i(T)) = s_i(T)$. Then $\pi_i\pi_{i+1}\pi_{i}(T)= \pi_i \pi_{i+1} (s_{i}(T)) = \pi_i(s_i(T)) = s_i(T)$, and $\pi_{i+1}\pi_{i}\pi_{i+1}(T) = \pi_{i+1}\pi_i(T) = \pi_{i+1}(s_i(T))=s_i(T)$. Finally, if $i+2$ is strictly right of $i$ in $T$ and not right-adjacent to $i$, then $i+2$ is strictly right of $i+1$ in $s_i(T)$ (and not right-adjacent to $i+1$), so $\pi_{i+1}(s_i(T)) = s_{i+1}s_i(T)$. Then $\pi_i\pi_{i+1}\pi_{i}(T)= \pi_i \pi_{i+1} (s_{i}(T)) = \pi_i(s_{i+1}s_i(T)) = s_{i+1}s_i(T)$, where the last equality is due to the fact that the cells occupied by $i$ and $i+1$ in $ s_{i+1}s_i(T)$ are those occupied by $i+1$ and $i+2$ respectively in $T$. On the other hand, we also have $\pi_{i+1}\pi_{i}\pi_{i+1}(T) = \pi_{i+1}\pi_i(T) = \pi_{i+1}(s_i(T))=s_{i+1}s_i(T)$. 

\noindent
(3)(c): Here we also have that $i+2$ is necessarily strictly right of $i$ and not right-adjacent to $i$ in $T$. As a result, we have  $\pi_i\pi_{i+1}\pi_i(T)=s_is_{i+1}s_i(T)$ and $\pi_{i+1}\pi_i\pi_{i+1}(T)=s_{i+1}s_is_{i+1}(T)$. Since $s_is_{i+1}s_i = s_{i+1}s_is_{i+1}$, these are the same tableaux.
\end{proof}

We now show that the $H_n(0)$-module ${\bf R}_\alpha$ has quasisymmetric characteristic $\R_\alpha$. Define a relation $\preceq$ on $\SYRT(\alpha)$ by declaring $T\preceq S$ if $S$ can be obtained from $T$ via applying a (possibly empty) sequence of the $\pi_i$ operators.

\begin{lemma}\label{lem:partialorder}
The relation $\preceq$ defines a partial order on $\SYRT(\alpha)$.
\end{lemma}
\begin{proof}
Reflexivity and transitivity of $\preceq$ are immediate from the definition. For antisymmetry, given $T\in \SYRT(\alpha)$ define a tuple $d(T)$ such that for each $1\le j \le {\rm max}(\alpha)$, the $j$th entry of $d(T)$ is the sum of the entries in the first $j$ columns of $T$. If $\pi_i(T) = s_i(T)$, then $i+1$ is strictly right of $i$ in $T$. Hence for each $1\le j \le {\rm max}(\alpha)$ we have $d(s_i(T))_j \ge d(T)_j$, and the inequality is strict in the column in which $i$ appears in $T$. Therefore, if $S\in \SYRT(\alpha)$ is obtained from $T$ by applying a sequence of the $\pi_i$, then either $S=T$ or else $d(S)_j > d(T)_j$ for some $1\le j \le {\rm max}(\alpha)$. In the latter case it is not possible to obtain $T$ from $S$ by applying operators $\pi_i$, as doing so can never decrease any entry of $d(S)$.
\end{proof}

We arbitrarily choose a total order $\preceq^\star$ on $\SYRT(\alpha)$ that extends the partial order $\preceq$. We may assume the elements of $\SYRT(\alpha)$ are ordered $T_m \preceq^\star T_{m-1} \preceq^\star \cdots \preceq^\star T_1$. For each $1\le j \le m$, let ${\bf R}_j={\rm span}\{T_1, \ldots , T_j\}$. Then for all $1\le j \le m$, ${\bf R}_j$ is a $H_n(0)$-submodule of ${\bf R}_\alpha$, and we have a filtration 
\[0:={\bf R}_0\subset {\bf R}_1 \subset {\bf R}_2 \subset \cdots \subset {\bf R}_m = {\bf R}_\alpha\]
of ${\bf R}_\alpha$.
It follows from the definition that each quotient module ${\bf R}_j/{\bf R}_{j-1}$ is one-dimensional, with basis $\{T_j\}$.

\begin{lemma}\label{lem:quotient}
For each $1\le i \le n-1$ and each $1\le j \le m$, in ${\bf R}_j/{\bf R}_{j-1}$ we have
\[\pi_i(T_j) = \begin{cases} T_j & \mbox{ if } i+1 \mbox{ is weakly left of } i \mbox{ in } T_j \\
                                             0  & \mbox{ otherwise. }
                                             \end{cases}
                                             \]
\end{lemma}
\begin{proof}
If $i+1$ is weakly left of $i$ in $T_j$, then $\pi_i(T_j)=T_j$ by Theorem~\ref{thm:0Hecke}. If $i+1$ is strictly right of $i$, then $\pi_i(T_j)$ is equal to either $0$ or $s_i(T_j)$ by Theorem~\ref{thm:0Hecke}. But $s_i(T_j)=0$ in ${\bf R}_j/{\bf R}_{j-1}$, since $s_i(T_j)\in {\bf R}_{j-1}$.
\end{proof}

\begin{theorem}
Let $\alpha\vDash n$. Then $ch([{\bf R}_\alpha]) = \R_\alpha$.
\end{theorem}
\begin{proof}
Each of the $H_n(0)$-modules ${\bf R}_j/{\bf R}_{j-1}$ is one-dimensional, thus irreducible, and therefore isomorphic to ${\bf F}_\beta$ for some composition $\beta$. It follows from Lemma~\ref{lem:quotient} that
\[\pi_i(T_j) = \begin{cases} T_j & \mbox{ if } i \notin \Des(T_j) \\
                                             0  & \mbox{ if } i \in \Des(T_j). 
                                             \end{cases}
                                             \]
By (\ref{eq:irreps}), this implies that ${\bf R}_j/{\bf R}_{j-1}$ is isomorphic as $H_n(0)$-modules to $\mathbf{F}_{\comp_n(\Des(T_j))}$, hence $[{\bf R}_j/{\bf R}_{j-1}]=[ \mathbf{F}_{\comp_n(\Des(T_j))}]$. It follows that
\[ch([{\bf R}_\alpha]) = \sum_{j=1}^m ch([{\bf R}_j/{\bf R}_{j-1}]) = \sum_{j=1}^m ch([\mathbf{F}_{\comp_n(\Des(T_j))}]) = \!\!\!\! \sum_{T\in \SYRT(\alpha)} \!\!\!\!  F_{\comp_n(\Des(T))} = \R_\alpha.\]
\end{proof}

\subsection{Direct sum decomposition}

A remaining goal is to classify for which $\alpha$ the $H_n(0)$-module ${\bf R}_\alpha$ is indecomposable. Towards this, we decompose ${\bf R}_\alpha$ into a direct sum of nonzero submodules, and show that each of these submodules is generated by a single $\SYRT$. We proceed in a similar manner to the work of \cite{TvW:1} that establishes analogous results for modules for quasisymmetric Schur functions.

Define a relation $\sim$ on $\SYRT(\alpha)$ by declaring $T\sim T'$ if for every $k$ such that $1\le k \le \max(\alpha)$, the relative order of the entries in the $k$th column of $T$ is the same as the relative order of the entries in the $k$th column of $T'$. It is immediate that $\sim$ is an equivalence relation, hence it gives rise to a partition of $\SYRT(\alpha)$. Suppose that $\sim$ decomposes $\SYRT(\alpha)$ into equivalence classes $E_0, E_1, \ldots , E_r$, where $E_0$ denotes the class consisting of all $T\in \SYRT(\alpha)$ such that entries increase from bottom to top in every column of $T$. 

\begin{example}
In Example~\ref{ex:SYRT}, the first two $\SYRT$s form the equivalence class $E_0$. The last three $\SYRT$s together form a different equivalence class: for each of these $\SYRT$s, the second column has the smallest entry at the bottom, second-smallest entry at the top, and largest entry in the middle; and in the first and third columns entries increase from bottom to top.
\end{example}

\begin{proposition}\label{prop:E0nonempty}
For any composition $\alpha$, the equivalence class $E_0$ is nonempty.
\end{proposition} 
\begin{proof}
Consider the tableau $\overline{T}$ of shape $\alpha$ formed by placing the entries $1, 2, \ldots , \alpha_1$ into row $1$, the entries $\alpha_1+1, \alpha_1+2 , \ldots , \alpha_2$ into row $2$, and so on. In Example~\ref{ex:SYRT}, this is the first tableau. By construction, in each column of $\overline{T}$ the entries are increasing from bottom to top. It is straightforward to see that $\overline{T}\in \SYRT(\alpha)$: (R1) and (R2) are satisfied by construction, and (R3) is satisfied because all entries of $\overline{T}$ in a higher row are greater than any entry in a lower row.
\end{proof}

Let ${\bf R}_\alpha^{E_j}$ denote the subspace of ${\bf R}_\alpha$ given by the complex span of $E_j$. 

\begin{proposition}
Let $\alpha\vDash n$. Then for each $j$, the vector space ${\bf R}_\alpha^{E_j}$ is an $H_n(0)$-submodule of ${\bf R}_\alpha$.
\end{proposition}
\begin{proof}
It suffices to show that for any $T\in E_j$ and any $1\le i \le n-1$, we have $\pi_i(T)\in {\bf R}_\alpha^{E_j}$. This is immediate when $\pi_i(T)=T$ or $\pi_i(T)=0$. If $\pi_i(T)=s_i(T)$, then $i$ and $i+1$ are necessarily in different columns of $T$. Replacing $i$ with $i+1$ in a column that does not contain $i+1$ does not change the relative order of entries in that column; likewise for replacing $i+1$ with $i$. Hence $s_i(T)\sim T$. 
\end{proof}

Consequently, we have

\begin{corollary}\label{cor:decomposition}
Let $\alpha\vDash n$. Then ${\bf R}_\alpha$ is isomorphic as $H_n(0)$-modules to $\bigoplus_{j=0}^r {\bf R}_\alpha^{E_j}$.
\end{corollary}

This implies the following result concerning indecomposability:

\begin{corollary}\label{cor:decomposable}
Let $\alpha\vDash n$. If there exists a $T\in \SYRT(\alpha)$ whose entries in some column do not increase from bottom to top, then ${\bf R}_\alpha$ is decomposable.
\end{corollary}
\begin{proof}
By Proposition~\ref{prop:E0nonempty}, the existence of such a $T$ implies there are at least two nonzero submodules in the expansion $\bigoplus_{j=0}^r {\bf R}_\alpha^{E_j}$. Thus by Corollary~\ref{cor:decomposition}, ${\bf R}_\alpha$ can be written as the direct sum of two nonzero submodules.
\end{proof}

This reduces the question of indecomposability to the case where ${\bf R}_\alpha = {\bf R}_\alpha^{E_0}$, i.e., for compositions $\alpha$ such that all $T\in \SYRT(\alpha)$ have entries increasing from bottom to top in every column. We now show that each equivalence class $E_j$ (and thus in particular $E_0$) contains a unique $T$ such that every $T'\in E_j$ can be obtained from $T$ by applying a sequence of the $\pi_i$ operators. As a result, ${\bf R}_\alpha^{E_j}$ is cyclic for each $j$, generated by this $T$.

Following the nomenclature of \cite{TvW:1} and \cite{Koenig}, we call $T\in E_j$ a \emph{source tableau} if there is no $T' \in E_j$ such that $T'\neq T$ and $\pi_i(T')=T$ for some $i$. We can characterize source tableaux as follows.

\begin{proposition}\label{prop:source}
Let $\alpha\vDash n$ and $T\in \SYRT(\alpha)$. Then $T$ is a source tableau if and only if for each entry $i<n$ such that $i\notin \Des(T)$, the entry $i+1$ is either in the same column as $i$, or in the column immediately left of the column containing $i$ and in a row higher than the row containing $i$.
\end{proposition}
\begin{proof}
Suppose there is an entry $i<n$ such that $i+1$ is at least two columns to the left of $i$, or in the column immediately left of $i$ and below $i$ (in which case it is strictly below $i$, by (R1)). Then $i$ and $i+1$ do not share a row or a column, and there is no triple involving both $i$ and $i+1$. Hence the tableau $T'$ obtained by exchanging $i$ and $i+1$ is an $\SYRT$, and we then have $\pi_i(T') = T$, so $T$ is not a source tableau.

Conversely, suppose that for each entry $i<n$ such that $i\notin \Des(T)$, the entry $i+1$ is either in the same column as $i$, or the column immediately left of the column containing $i$ and a row higher than the row containing $i$. If there were an $\SYRT$ $T'\neq T$ such that $\pi_i(T')=T$ for some $i$, then $i+1$ must be strictly left of $i$ in $T$, and $T'$ is obtained from $T$ by exchanging $i$ and $i+1$. Then in $T$, $i+1$ must be in the column immediately left of the column containing $i$, and above $i$. But this means $i+1$ and $i$ form two cells of a triple in $T$, and exchanging them violates (R3) since the entry of the cell right-adjacent to $i+1$, if it exists must be larger than $i+1$ by (R1), and thus larger than $i$. Hence there is no $T'\neq T$ such that $\pi_i(T')=T$ for some $i$, and thus $T$ is a source tableau.
\end{proof}

For example, the second and the fourth $\SYRT$s in Example~\ref{ex:SYRT} are the source tableaux of shape $(3,2,2)$. Notice these source tableaux belong to different equivalence classes.

\begin{lemma}\label{lem:sourceexists}
For any composition $\alpha$, there exists at least one source tableau in each equivalence class in $\SYRT(\alpha)$. 
\end{lemma}
\begin{proof}
Recall from Lemma~\ref{lem:partialorder} that $\SYRT(\alpha)$ is partially ordered by the relation $\preceq$, and thus the tableaux in $E_j$ are partially ordered by $\preceq$. The source tableaux are, by definition, the minimal elements in $E_j$ under $\preceq$. Minimal elements must exist since $\SYRT(\alpha)$, and thus $E_j$, is a finite set.
\end{proof}

We now show that each equivalence class $E_j$ contains exactly one source tableau. Given a composition $\alpha$, define a cell of $D(\alpha)$ to be \emph{removable} if it is the rightmost cell of the top row, or the rightmost cell in a row of length at least $2$ that has no row above it containing precisely one fewer cell. Given $T\in \SYRT(\alpha)$, call a removable cell of $D(\alpha)$ a \emph{distinguished removable cell} of $T$ if this cell contains the largest entry in its column. 

\begin{example}\label{ex:removable}
For $\alpha = (2,3,4,2)$, the diagram on the left shows the removable cells of $D(\alpha)$, and the diagram on the right shows an $\SYRT$ of shape $\alpha$ with distinguished removable cells indicated by bolded entries.
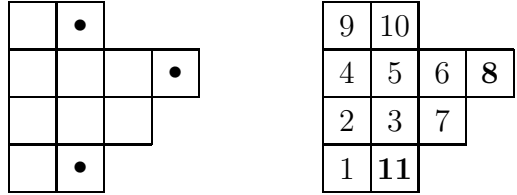
\[ \tableau{ {\ } & \bullet \\ {\ } & {\ } & {\ } & \bullet \\  {\ } & {\ } & {\ } \\ {\ } & \bullet } \qquad \qquad \tableau{ 9 & 10 \\ 4 & 5 & 6 & {\bf 8} \\  2 & 3 & 7 \\ 1 & {\bf 11} } \]
\end{example}

The importance of removable cells stems from the following result.

\begin{lemma}\label{lem:nremovable}
Let $\alpha\vDash n$. Then a cell $\kappa$ in $D(\alpha)$ is removable if and only if there is some $T\in \SYRT(\alpha)$ with entry $n$ in $\kappa$. Moreover, the cell with entry $n$ in any $T\in \SYRT(\alpha)$ is a distinguished removable cell.
\end{lemma}
\begin{proof}
Suppose $\kappa$ is a removable cell in row $i$ of $D(\alpha)$. Consider any $\SYRT$ of shape $\hat{\alpha} = (\alpha_1, \ldots , \alpha_i-1, \ldots , \alpha_{\ell(\alpha)})$, and fill the cells of $D(\alpha)$ with this $\SYRT$ filling in cells other than $\kappa$ and $n$ in $\kappa$. We claim this filling of $D(\alpha)$ is an $\SYRT$. Certainly (R1) and (R2) are satisfied. Since $\kappa$ is removable, it is the lower cell in a triple only if both upper cells are in $D(\alpha)$. This means the upper cells will both have entries smaller than $n$, thus (R3) is also satisfied.

Conversely, suppose $\kappa$ is not removable. If $\kappa$ is not at the end of a row, then by (R1) it cannot have entry $n$ in any $\SYRT$. If $\kappa$ is at the end of row $i$ but there is a row $j$ above row $i$ exactly one cell shorter than row $i$, then filling the last cell of row $i$ with $n$ would cause a violation of (R3) with the entry of the last cell in row $j$.

Finally, since $n$ is necessarily the largest entry in its column, the removable cell occupied by $n$ is distinguished.
\end{proof}

Since all $T\in E_j$ have the same relative order of entries in each column, they have the same distinguished removable cells. Therefore, define ${\rm DR}(E_j)$ to be the set of column indices of $\SYRT$s in $E_j$ that contain a distinguished removable cell. For example, the $\SYRT$ in Example~\ref{ex:removable} has distinguished removable cells in columns $2$ and $4$, and thus its equivalence class $E_j$ has ${\rm DR}(E_j)=\{2,4\}$.

\begin{lemma}\label{lem:positionofn}
Let $\alpha\vDash n$, let $E_j$ be an equivalence class for $\SYRT(\alpha)$, and let $T\in E_j$ be a source tableau. If $M$ is the largest element of ${\rm DR}(E_j)$, then the distinguished removable cell in column $M$ of $T$ has entry $n$.
\end{lemma}
\begin{proof}
By Lemma~\ref{lem:sourceexists} there is at least one source tableau in $E_j$. Moreover by Lemma~\ref{lem:nremovable}, we know the index of the column of $T$ containing $n$ is an element of ${\rm DR}(E_j)$. Suppose the index of this column is $N<M$, and thus we have some entry $i<n$ in the distinguished removable cell in column $M$. We will show that $T$ cannot be a source tableau. Define a set
\[\mathcal{S} = \{k\in \mathbb{N} : i < k \le n \mbox{ and $k$ is in a column strictly left of column $M$} \}.\]

First suppose $i+1\in \mathcal{S}$. Then $i+1$ is strictly left of $i$ in $T$, and thus in a different row than $i$ by (R1). If $i+1$ is at least two columns to the left of $i$, or in the column immediately left of $i$ and in a row below $i$, then the tableau $T'$ obtained from $T$ by exchanging $i$ and $i+1$ is in $E_j$, and we have $\pi_i(T') = s_i(T') = T$, so $T$ is not a source tableau. We conclude this case by noting that $i+1$ cannot be in the column immediately left of $i$ and above $i$. Suppose it were. Then $i+1$ cannot be the last entry in its row, as this would contradict $i$ occupying a removable cell. But if $i+1$ is not the last entry in its row, then the entry right-adjacent to $i+1$ is greater than $i+1$ by (R1) and in the same column as $i$, contradicting that $i$ is in a \emph{distinguished} removable cell. 

Now suppose $i+1\notin \mathcal{S}$. Let $\delta$ denote the minimum element of $\mathcal{S}$ (note $n\in \mathcal{S}$, so $\mathcal{S}$ is nonempty and thus has a minimum element). Notice that since $\delta-1 \ge i+1$, we have $\delta-1\notin \mathcal{S}$, that is, $\delta-1$ is weakly right of $i$. In fact $\delta-1$ is strictly right of $i$, since $i<\delta-1$ is the largest entry in its column. On the other hand $\delta$ is strictly left of $i$, so $\delta$ is at least two columns left of $\delta-1$. Therefore, by (R1), $\delta$ and $\delta-1$ are in different rows. It follows that the tableau $T'$ obtained by swapping $\delta$ and $\delta-1$ in $T$ is in $E_j$, and then $\pi_{\delta-1}(T') = T$, so $T$ is not a source tableau.
\end{proof}

\begin{corollary}\label{cor:sourceunique}
Let $\alpha\vDash n$. Every equivalence class $E_j$ for $\SYRT(\alpha)$ has a unique source tableau.
\end{corollary}
\begin{proof}
By Lemma~\ref{lem:sourceexists}, we know that $E_j$ has at least one source tableau. To prove it is unique, we proceed by induction on $n=|\alpha|$. If $n=1$, then $\SYRT(\alpha)$ has only one element, and thus no more than one source tableau. Now let $n>1$ and suppose that for every $\alpha'\vDash n-1$ each equivalence class for $\SYRT(\alpha')$ has a unique source tableau. Suppose further that $T, S\in E_j\subseteq \SYRT(\alpha)$ are source tableaux. Then by Lemma~\ref{lem:positionofn}, $n$ occupies the same cell in $T$ as it does in $S$. 

Let $T'$ and $S'$ denote the tableaux obtained by deleting the cell with entry $n$ from $T$ and $S$. It is clear that $T'$ and $S'$ are source tableaux for a composition of $n-1$, and belong to the same equivalence class. Hence $T'=S'$ by the inductive hypothesis, and therefore $T=S$.
\end{proof}

The next result follows immediately from Corollary~\ref{cor:sourceunique}

\begin{corollary}\label{cor:cyclic}
Each submodule ${\bf R}_\alpha^{E_j}$ of ${\bf R}_\alpha$ is cyclic, generated by the unique source tableau in $E_j$. 
\end{corollary}

\subsection{Simple compositions}

Before addressing indecomposability, we characterize the compositions that give rise to only a single equivalence class of $\SYRT$s. Following the nomenclature of \cite{TvW:1}, define a composition $\alpha$ to be \emph{simple} if whenever $\alpha_j\ge \alpha_i \ge 2$ for some $1\le i <j \le \alpha$, there is some $k$ such that $i \le k \le j$ and $\alpha_k = \alpha_i-1$. In other words, given a pair of rows in $D(\alpha)$ where the lower row is weakly shorter (and of length at least 2), there is another row weakly between this pair of rows that is one cell shorter than the lower one.

\begin{example}
The compositions $(2,1,1,3)$ and $(4,2,1,2)$ (left) are simple, whereas $(2,3,1,4)$ and $(3,3,3,1)$ (right) are not.
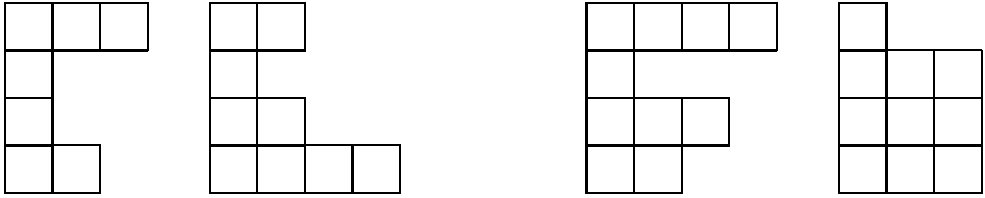
\[ \tableau{ {\ } & {\ } & {\ }  \\  {\ } \\ {\ } \\ {\ } & {\ }  \\ } \qquad  \tableau{ {\ } & {\ }  \\  {\ } \\ {\ } & {\ } \\ {\ } & {\ } & {\ } & {\ } \\ } \qquad \qquad \qquad \tableau{ {\ } & {\ } & {\ } & {\ } \\  {\ } \\   {\ } & {\ } & {\ } \\ {\ } & {\ } \\ } \qquad \tableau{ {\ } \\  {\ } & {\ } & {\ } \\  {\ } & {\ } & {\ } \\ {\ } & {\ } & {\ } \\  } \]
\end{example}

\begin{remark}
Compositions that give rise to a single equivalence class in the modules for quasisymmetric Schur functions are classified in \cite{TvW:1}. As mentioned in the introduction, the above condition for simplicity of a composition is the same as that in \cite{TvW:1}, up to reversal. 
This is due to the fact that in both cases the number of equivalence classes depends only on $\alpha$ and the definitions of the standard tableaux indexing the fundamental expansion, which are similar. 
Accordingly, the argument below that simple compositions characterise existence of only a single equivalence class proceeds similarly to the analogous argument in \cite{TvW:1}. On the other hand, the descent structure of $\SYRT$s differs from that of the standard composition tableaux defining quasisymmetric Schur functions, hence the actual proof of indecomposability in Section~\ref{sec:indecomposability} differs considerably from that in \cite{TvW:1}; see Remark~\ref{rmk:indecomposability}.
\end{remark}

\begin{lemma}\label{lem:simpleinduct}
Let $\alpha\vDash n$ be simple, with $D(\alpha)$ having a removable cell in row $i$. Then $\hat{\alpha} = (\alpha_1 ,\ldots , \alpha_i-1,\ldots , \alpha_{\ell(\alpha)})$ is also simple.
\end{lemma}
\begin{proof}
We check that each pair of rows in $D(\hat{\alpha})$ satisfies the condition for simplicity. Clearly this condition is satisfied for any pair of rows that are both above or both below row $i$, since these rows and all rows between them are the same as in $D(\alpha)$. For a pair consisting of row $i$ and some lower row, the condition is clearly satisfied since if row $i$ of $D(\hat{\alpha})$ is weakly longer than the lower row, then certainly row $i$ of $D(\alpha)$ is also longer than the lower row, and the condition was satisfied for these rows in $D(\alpha)$.

Each row above row $i$ in $D(\alpha)$ is at least two cells shorter than row $i$. This is because there cannot exist a row above row $i$ that is exactly one cell shorter, since row $i$ has a removable cell; and the existence of a row above row $i$ that was weakly longer would imply (by simplicity of $\alpha$) the existence of a row above row $i$ that was exactly one cell shorter than row $i$, a contradiction. 

Hence for all $j>i$, we have $\hat{\alpha}_j<\hat{\alpha}_i$, so every pair of rows involving row $i$ and a row above it in $D(\hat{\alpha})$ satisfies the condition to be simple. The remaining case is a pair of rows with one strictly below and one strictly above row $i$. The only way this pair could fail the condition is if $\alpha_i$ was the only part of $\alpha$ satisfying the condition for this pair in $D(\alpha)$, that is, the higher row in the pair is weakly longer than the lower row and the $i$th row is one cell shorter than the lower row. But this would mean that row $i$ in $\hat{\alpha}$ is shorter than a row above it (namely, the higher in the pair), which we have seen is impossible.
\end{proof}

\begin{proposition}\label{prop:simple}
A composition $\alpha$ is simple if and only if for every $T\in \SYRT(\alpha)$, entries increase from bottom to top in each column of $T$.
\end{proposition}
\begin{proof}
Suppose $\alpha\vDash n$ is simple. We proceed by induction on $n$; the base case where $n=1$ is clear. Suppose that entries increase from bottom to top in all columns for every simple composition of $n-1$. Consider the entry $n$ in $T$. By Lemma~\ref{lem:nremovable} the cell containing $n$ is always removable, and then, since $\alpha$ is simple, this cell is the highest in its column (the existence of a weakly longer row above this cell would imply the existence of another row above that is one cell shorter, contradicting removability). Now delete this cell; by Lemma~\ref{lem:simpleinduct} the resulting tableau is an $\SYRT$ for a simple composition of $n-1$. By the inductive hypothesis, entries all columns in this $\SYRT$ increase from bottom to top. Then since $n$ is highest in its column in $T$, entries in all columns of $T$ increase from bottom to top.

Conversely, suppose $\alpha$ is not simple. Then there exist a pair of rows of $D(\alpha)$ with indices $i<j$, each of length at least 2, such that the higher row $j$ is weakly longer than the lower row $i$ and there is no row between them that is one cell shorter than row $i$. Define compositions $\alpha_{{\rm low}} = (\alpha_1, \ldots , \alpha_j)$ and $\alpha_{{\rm high}} = (\alpha_{j+1}, \ldots , \alpha_{\ell(\alpha)})$. Then in $D(\alpha_{{\rm low}})$, the rightmost cell in row $i$ is removable. Choose any $T_{{\rm low}}\in \SYRT(\alpha_{{\rm low}})$ such that the largest entry $|\alpha_{{\rm low}}|$ in $T_{{\rm low}}$ is in this removable cell (this can be done by Lemma~\ref{lem:nremovable}). Observe that this removable cell is not the highest cell in its column, hence entries do not increase upwards in this column. Now choose any filling $T_{{\rm high}}$ of $D(\alpha_{{\rm high}})$ with the entries $|\alpha_{{\rm low}}|+1, \ldots , n$ that satisfies the $\SYRT$ conditions. Then the filling of $D(\alpha)$ whose lowest $j$ rows are filled as $T_{{\rm low}}$ and remaining rows are filled as $T_{{\rm high}}$ is an $\SYRT$, and has a column in which entries do not increase from bottom to top.  
\end{proof}

\section{Classification of indecomposability}\label{sec:indecomposability}

In this section we establish the following theorem, classifying when ${\bf R}_\alpha$ is indecomposable.

\begin{theorem}\label{thm:main}
The $H_n(0)$-module ${\bf R}_\alpha$ is indecomposable if and only if $\alpha$ is simple.  
\end{theorem}

One direction is immediate from the results in Section~\ref{sec:modules}. By Corollary~\ref{cor:decomposable}, ${\bf R}_\alpha$ is decomposable whenever $\SYRT(\alpha)$ has more than one equivalence class. Therefore by Propositions~\ref{prop:E0nonempty} and \ref{prop:simple}, if $\alpha$ is not simple, then ${\bf R}_\alpha$ is decomposable.

To prove the converse direction, we will show more generally that for any $\alpha\vDash n$, the submodule ${\bf R}_\alpha^{E_0}$ of  ${\bf R}_\alpha$ is indecomposable. Then by Proposition~\ref{prop:simple}, it follows that when $\alpha$ is simple, ${\bf R}_\alpha = {\bf R}_\alpha^{E_0}$ is indecomposable.

\subsection{The source tableau of $E_0$}
We begin by establishing a concrete description of the source tableau of $E_0$, which will be needed later. Let $\alpha\vDash n$ and define the \emph{boundary cells} of $D(\alpha)$ to be the cells in the first column, along with the cells that have no cell strictly above them in the same column or in the column immediately to the left. Order the boundary cells by $(a,b)<(c,d)$ if either $a=c=1$ and $b<d$, or $a<c$. Note that this total order proceeds up the first column, then rightwards. 

To each boundary cell we associate a collection of cells in $D(\alpha)$ called a \emph{thread}. We say a cell is \emph{threaded} if it (already) belongs to a thread. The thread associated to the first boundary cell $(1,1)$ is just the cell $(1,1)$ itself. Assuming threads have been associated to the first $j-1$ boundary cells, the thread associated to the $j$th boundary cell consists of the $j$th boundary cell $\kappa$, the highest unthreaded cell strictly below $\kappa$ in the column immediately to the right of $\kappa$, the highest unthreaded cell strictly below that in the next column to the right, and so on. The thread terminates when there is no unthreaded cell strictly below in the next column to the right. In this way, each thread is a sequence of cells in consecutive columns, proceeding strictly northwest to southeast in $D(\alpha)$, and each cell belongs to at most one thread.

\begin{example}\label{ex:thread}
For $\alpha = (2,5,1,3,3)$, we label the cells of $D(\alpha)$ according to their thread: the $j$th thread consists of all cells with entry $j$. The entry in each boundary cell is bolded. 
\[\tableau{ {\bf 5} & {\bf 6} & {\bf 7} \\ {\bf 4} & 5 & 6 \\ {\bf 3} \\ {\bf 2} & 3 & 5 & 6 & {\bf 8} \\ {\bf 1} & 2 }\] 
\end{example}

\begin{lemma}
For any composition $\alpha$, the threads partition $D(\alpha)$.
\end{lemma}
\begin{proof}
By definition, each cell belongs to at most one thread. To show each cell belongs to some thread, suppose for a contradiction that some cells were not threaded during the threading process. Consider the leftmost, then highest such cell; call it $\kappa$ and suppose it is in column $c$. Since all boundary cells are threaded by definition, $\kappa$ cannot be a boundary cell. Hence there exists a cell in $D(\alpha)$ strictly above $\kappa$ in column $c-1$. Let $\kappa'$ be the lowest such cell. By assumption, $\kappa'$ is threaded. Since there exist unthreaded cells in column $c$ strictly below $\kappa'$ at the time $\kappa'$ gets threaded (in particular, $\kappa$ is such a cell), the thread containing $\kappa'$ continues to column $c$. But $\kappa$ is the highest cell in column $c$ strictly below $\kappa'$, and by assumption is unthreaded at the time the thread containing $\kappa'$ is being created, so the thread containing $\kappa'$ must continue to $\kappa$, contradicting that $\kappa$ is unthreaded.
\end{proof}

\begin{lemma}\label{lem:threadedSW}
For any composition $\alpha$, there is never an unthreaded cell weakly southwest of a threaded cell in $D(\alpha)$ at any point during the threading process. 
\end{lemma}
\begin{proof}
We claim a thread always takes the \emph{lowest} unthreaded cell in each column, which immediately implies the statement. Suppose when constructing the $j$th thread, we did not take the lowest unthreaded cell in some column $c$, and that this was the first instance in the threading process that a non-lowest cell was taken. Let $\kappa$ denote the lowest unthreaded cell in column $c$ at this instance in the process. We may assume $c>1$, since all cells in the first column are boundary cells and the threading process forces each of them to be threaded in order from bottom to top. Consider the lowest cell $\kappa'$ in column $c-1$ that is strictly above $\kappa$. Note that the $j$th thread has used a cell in column $c$ strictly above $\kappa$, so it must have used a cell in column $c-1$ strictly above $\kappa$ as well. By our assumption, there is currently no threaded cell above an unthreaded cell in column $c-1$, so in particular, the cell $\kappa'$ must already be threaded. But if $\kappa'$ is the lowest cell in column $c-1$ that is strictly above $\kappa$, then, since rows of $D(\alpha)$ are left-justified, $\kappa$ is the highest cell in column $c$ that is strictly below $\kappa'$. Hence, since  $\kappa$ is currently unthreaded, $\kappa'$ and $\kappa$ must belong to the same thread. If the thread of $\kappa'$ is not the $j$th thread, 
this contradicts $\kappa$ being currently unthreaded. If the thread of $\kappa'$ is the $j$th thread, this contradicts our assumption that the $j$th thread chooses a different cell to $\kappa$ in column $c$.
\end{proof}

Suppose $D(\alpha)$ has threads $L_1, \ldots , L_m$, in order. Define a standard filling $T_{\sup}$ of $D_\alpha$ by filling each thread $L_k$ with $|L_1|+\ldots + |L_k-1| + 1, |L_1|+\ldots + |L_k-1| + 2, \ldots , |L_1|+\ldots + |L_k-1| + |L_k|$ consecutively from right to left. 

\begin{example}\label{ex:superstandard}
Let $\alpha = (2,5,1,3,3)$, as in Example~\ref{ex:thread}. Then
\[T_{\sup} = \tableau{ 9 & 12 & 13 \\ 6 & 8 & 11 \\ 5 \\ 3 & 4 & 7 & 10 & 14 \\ 1 & 2 }.\] 
\end{example}

\begin{proposition}\label{prop:fillingisT0}
For any composition $\alpha$, $T_{\sup}$ is the source tableau of $E_0$.
\end{proposition}
\begin{proof}
It follows immediately from Lemma~\ref{lem:threadedSW} that $T_{\sup}$ satisfies (R1) and (R2), and that entries increase from bottom to top in all columns. For (R3), it is enough to confirm that for any pair consisting of a cell with entry $x$ in column $c-1$ strictly above a cell with entry $z$ in column $c$, we have $x>z$. Suppose for a contradiction that we had $x<z$. Then the cell with entry $z$ cannot belong to the same thread as the cell with entry $x$; it belongs to a later thread. This means that when constructing the thread that uses the cell with entry $x$, the cell with entry $z$ was unthreaded but not used by this thread. Since the cell with entry $z$ is moreover in the column immediately right of the cell with entry $x$ and strictly below it, the thread using the cell with entry $x$ must have used a cell in that column. By Lemma~\ref{lem:threadedSW}, this cell cannot be strictly above the cell with entry $z$, but by the definition of threading this cell cannot be strictly below the cell with entry $z$. This contradicts $z$ being greater than $x$. Hence $T_{\sup}\in E_0$.

It remains to show $T_{\sup}$ is a source tableau; by Corollary~\ref{cor:sourceunique}, this implies it is the only source tableau of $E_0$. Consider any entry $i$. By definition of $T_{\sup}$ either $i+1$ is in the column immediately left of $i$ and above $i$, or $i$ occupies a boundary cell and $i+1$ is the rightmost entry of the next thread. In the latter case, $i+1$ cannot be left of $i$, since the sequence of boundary cells proceeds weakly leftwards and all cells in a thread are to the right of the boundary cell of that thread. Hence the condition in Proposition~\ref{prop:source} is satisfied.  
\end{proof}

\subsection{Proof of indecomposability of ${\bf R}_\alpha^{E_0}$}

A module $M$ is indecomposable if and only if the only idempotent module endomorphisms of $M$ are $0$ and $1$ \cite{Jacobson}. Suppose $f$ is an idempotent $H_n(0)$-module morphism of ${\bf R}_\alpha^{E_0}$. By Corollary~\ref{cor:cyclic}, ${\bf R}_\alpha^{E_0}$ is generated by $T_{\sup}$, thus $f: {\bf R}_\alpha^{E_0} \rightarrow  {\bf R}_\alpha^{E_0}$ is completely determined by $f(T_{\sup})$. Let 
\[f(T_{\sup}) = \sum_{T\in E_0} a_T T.\] 
We will show that in fact $a_T=0$ for all $T\neq T_{\sup}$; it then follows that $f(T_{\sup}) = a_{T_{\sup}} T_{\sup}$, whence idempotence of $f$ immediately implies that $f$ is either $0$ or $1$. 

The following lemma establishes that $a_{T'}=0$ for a large class of $\SYRT$s $T'\in E_0$.

\begin{lemma}\label{lem:descentofTprime}
Let $T'\in E_0$. If there an $i$ such that $i\in \Des(T')$ but $i\notin \Des(T_{\sup})$, then $a_{T'}=0$.
\end{lemma}
\begin{proof}
Since $\pi_i(T_{\sup})=T_{\sup}$, we have 
\[f(T_{\sup}) = f(\pi_i(T_{\sup})) = \pi_i(f(T_{\sup})) = \pi_i( \sum_{T\in \SYRT(\alpha)}a_T T) =  \sum_{T\in \SYRT(\alpha)}a_T \pi_i(T).\]
Therefore the coefficient $a_{T'}$ of $T'$ in $f(T_{\sup})$ is the sum of the coefficients of the $S\in \SYRT(\alpha)$ such that $\pi_i(S)=T'$. But $\pi_i(T')\neq T'$ since $i\in \Des(T')$. Therefore if $\pi_i(S) = T'$, then $T' = \pi_i(S) = \pi_i^2(S) = \pi_i(T') \neq T'$, a contradiction. So there is no such $S$, and thus $a_{T'}=0$. 
\end{proof}

\begin{remark}\label{rmk:indecomposability}
For dual immaculate quasisymmetric functions, quasisymmetric Schur functions and extended Schur functions, the indecomposability classification follows immediately from the appropriate analogue of Lemma~\ref{lem:descentofTprime} (\cite{BBSSZ}, \cite{TvW:1}, \cite{Sea20}). Specifically, for each of these families of functions, the source tableau of the relevant cyclic $0$-Hecke (sub)module has an especially simple form, namely, the filling we use in the proof of Proposition~\ref{prop:E0nonempty} or a reversal of this. This can then be used to show that \emph{every} non-source tableau has some descent that is not a descent of the source tableau. 
On the other hand, for ${\bf R}_\alpha^{E_0}$ the source tableau $T_{\sup}$ is more complicated, and indeed not every $T\in E_0$ has a descent that is not a descent of the source tableau, even if $\alpha$ is simple. Therefore establishing indecomposability of ${\bf R}_\alpha^{E_0}$ requires further analysis.
\end{remark}

From now on, fix $\hat{T}\in E_0$ such that $\hat{T}\neq T_{\sup}$ and $\Des({\hat{T}})\subseteq \Des(T_{\sup})$. Lemma~\ref{lem:descentofTprime} reduces the problem to showing that $a_{\hat{T}} = 0$.  
To do this, we make use of a technique of \cite{Koenig}, which requires us to establish the existence of a sequence of operators that sends $T_{\sup}$ to $0$ but does not send $\hat{T}$ to $0$, such that each operator in the sequence exchanges entries of $\hat{T}$. We exhibit such a sequence in Corollary~\ref{cor:main}. 

Fix a sequence of operators $\pi_{i_1} \ldots \pi_{i_p}$ such that $\pi_{i_1} \ldots \pi_{i_p}(T_{\sup}) = s_{i_1}\ldots s_{i_p}(T_{\sup})= \hat{T}$. Such a sequence exists since $T_{\sup}$ is the source tableau of $E_0$ (Proposition~\ref{prop:fillingisT0}). Let $\varepsilon$ denote the smallest entry that occupies a different cell in $\hat{T}$ to the cell it occupies in $T_{\sup}$. For the following lemmas leading to Corollary~\ref{cor:main}, we use the following running example as an illustration.

\begin{example}\label{ex:TsupThat}
Let $\alpha = (5,3,4,1,2)\vDash 15$. Below are $T_{\sup}$  and a $\hat{T}\in E_0$ with $\Des(\hat{T}) = \{1,3,6,11, 13\} = \Des(T_{\sup})$.
\[T_{\sup} = \tableau{ 11 & 13 \\ 10  \\ 6 & 9 & 12 & 15  \\ 3 & 5 & 8  \\ 1 & 2 & 4 & 7 & 14 \\ } \qquad \qquad \hat{T} = \tableau{ 11 & 15 \\ 10 \\ 6 & 9 & 13 & 14 \\ 3 & 5 & 8 \\ 1 & 2 & 4 & 7 & 12 \\ }\]
Here we have $\pi_{14}\pi_{12}\pi_{13}(T_{\sup})= s_{14}s_{12}s_{13}(T_{\sup}) = \hat{T}$ and $\varepsilon = 12$. Notice this sequence of operators does not contain any $\pi_i$ such that $i<\varepsilon$, and that $\varepsilon$ occupies a cell in $\hat{T}$ strictly right of the cell it occupies in $T_{\sup}$, agreeing with Lemmas~\ref{lem:fixedsubset} and \ref{lem:epsilonright} below.
\end{example}
 
\begin{lemma}\label{lem:fixedsubset}
For any $i<\varepsilon$, $\pi_i$ does not appear in any sequence $\pi_{i_1} \ldots \pi_{i_p}$ of operators such that $\pi_{i_1} \ldots \pi_{i_p}(T_{\sup}) =s_{i_1} \ldots s_{i_p}(T_{\sup}) = \hat{T}$. 
\end{lemma} 
\begin{proof}
We proceed by induction on $i$. Since $1$ is always in the lowest cell in the first column of any $\SYRT$, $\pi_1$ cannot act as $s_1$ on any $\SYRT$, thus $\pi_1$ does not appear in the sequence. Now let $1<i<\varepsilon$ and suppose $\pi_1, \ldots , \pi_{i-1}$ do not appear. By definition, application of $\pi_i$ moves $i$ strictly rightwards, and the only way to move $i$ strictly leftwards is by applying $\pi_{i-1}$. Since by assumption $\pi_{i-1}$ is never applied, if $\pi_i$ is applied then in $\hat{T}$ the entry $i$ occupies a position strictly right of the position $i$ occupies in $T_{\sup}$, contradicting the minimality of $\varepsilon$. 
\end{proof}

\begin{lemma}\label{lem:epsilonright}
The cell of $D(\alpha)$ occupied by $\varepsilon$ in $\hat{T}$ is strictly right of the cell of $D(\alpha)$ occupied by $\varepsilon$ in $T_{\sup}$.
\end{lemma} 
\begin{proof}
By Lemma~\ref{lem:fixedsubset}, $\pi_{\varepsilon-1}$ never occurs in $\pi_{i_1} \ldots \pi_{i_p}$. However, since $\varepsilon$ occupies a different cell in $\hat{T}$ than it does in $T_{\sup}$, $\pi_\varepsilon$ must occur. Application of $\pi_{\varepsilon}$ moves $\varepsilon$ strictly rightwards, and $\varepsilon$ cannot ever move strictly leftwards because the sequence does not contain  $\pi_{\varepsilon-1}$.
\end{proof}

In Example~\ref{ex:TsupThat}, notice that in $T_{\sup}$ the entry $\varepsilon=12$ occupies the rightmost cell in its thread (that is, the cells with entries $12$ and $13$), and that $\varepsilon-1=11$ is a descent, agreeing with the statements of Lemmas~\ref{lem:rightmostinthread} and \ref{lem:epsilonminusonedescent} below.

\begin{lemma}\label{lem:rightmostinthread}
The cell containing $\varepsilon$ in $T_{\sup}$ is the rightmost cell in its thread.
\end{lemma}
\begin{proof}
Suppose the cell containing $\varepsilon$ in $T_{\sup}$ was not the rightmost in its thread. Then $\varepsilon-1$ is in the column immediately right of the column containing $\varepsilon$; in particular, $\varepsilon-1$ is not a descent in $T_{\sup}$. We will show that $\varepsilon-1$ must be a descent in $\hat{T}$, contradicting the assumption $\Des(\hat{T})\subseteq \Des(T_{\sup})$. 

By Lemma~\ref{lem:fixedsubset}, $\varepsilon-1$ occupies the same cell in $\hat{T}$ as it does in $T_{\sup}$, and by Lemma~\ref{lem:epsilonright} $\varepsilon$ occupies a cell in $\hat{T}$ that is strictly right of the cell it occupies in $T_{\sup}$. Therefore, it suffices to show that $\varepsilon$ is moved at least two columns rightwards when $\pi_\varepsilon$ is first applied, since this will ensure $\varepsilon$ is strictly right of $\varepsilon-1$ in $\hat{T}$. (Any subsequent applications of $\pi_\varepsilon$ only move $\varepsilon$ further rightwards.)  When $\pi_\varepsilon$ is applied, $\varepsilon$ swaps with $\varepsilon+1$, so we need to show that $\varepsilon+1$ cannot be in the column immediately right of the column that $\varepsilon$ occupies in $T_{\sup}$ when $\pi_\varepsilon$ is first applied, since then $\varepsilon$ would only move one column rightwards.

In $T_{\sup}$, there is no cell in the column of $\varepsilon-1$ (i.e., the column immediately right of the column of $\varepsilon$) that is above $\varepsilon-1$ and strictly below $\varepsilon$. (Otherwise since columns increase from bottom to top, this cell would have an entry greater than $\varepsilon$, and this cell and the cell containing $\varepsilon$ form a triple which would violate (R3)). 
Moreover any cell below $\varepsilon-1$ in the column of $\varepsilon-1$ has an entry smaller than $\varepsilon-1$ due to the increasing column condition, and these entries never change during the process due to Lemma~\ref{lem:fixedsubset}. 

Therefore when $\pi_\varepsilon$ is first applied, if $\varepsilon+1$ is in the column of $\varepsilon-1$, it must be weakly above $\varepsilon$. Let $T'\in E_0$ denote the $\SYRT$ to which $\pi_\varepsilon$ is first applied in the process. It is impossible for $\varepsilon+1$ to be strictly above $\varepsilon$ in $T'$, because then the entry in the cell in the column containing $\varepsilon$ and the row containing $\varepsilon+1$ would have to be both strictly greater than $\varepsilon$ and strictly smaller that $\varepsilon+1$ due to (R1) and the increasing column condition. Therefore, $\varepsilon+1$ can only be in the column immediately right of $\varepsilon$ in $T'$ if it is right-adjacent to $\varepsilon$, but then $\pi_\varepsilon(T')=0$, contradicting that $\pi_{i_1} \ldots \pi_{i_p}(T_{\sup}) = \hat{T}$. 
\end{proof}

\begin{lemma}\label{lem:epsilonminusonedescent}
The entry $\varepsilon-1$ is a descent of $T_{\sup}$.
\end{lemma}
\begin{proof}
By Lemma~\ref{lem:rightmostinthread}, $\varepsilon$ is the rightmost (and thus smallest) entry in its thread. Hence $\varepsilon-1$ is the largest (and leftmost) entry in the preceding thread.  Since later threads start weakly right of earlier threads and cells in a thread proceed strictly rightwards, if the thread containing $\varepsilon$ has at least $2$ cells, then we are done. If it has only one cell, then the only way for $\varepsilon$ to not be strictly right of $\varepsilon-1$ in $T_{\sup}$ is for both $\varepsilon$ and $\varepsilon-1$ to be in threads consisting of a single cell in the first column. But when transforming $\hat{T}$ into $T_{\sup}$, $\pi_{\varepsilon-1}$ is never applied whereas $\pi_{\varepsilon}$ is. Since applying $\pi_{\varepsilon}$ moves $\varepsilon$ strictly rightwards, $\varepsilon-1$ is a descent in $\hat{T}$, contradicting $\Des(\hat{T})\subseteq \Des(T)$.
\end{proof}

Lemma~\ref{lem:epsilonminusonedescent} implies that $\varepsilon$ is not in the first column of $T_{\sup}$, since $\varepsilon$ must be strictly right of $\varepsilon-1$. Therefore, in $T_{\sup}$ there exists a cell left-adjacent to the cell containing $\varepsilon$. Let $x$ denote the entry left-adjacent to $\varepsilon$ in $T_{\sup}$. In Example~\ref{ex:TsupThat}, we have $x=9$. Note also that the entries $9,10,11$ are all strictly left of $\varepsilon = 12$ in both $T_{\sup}$ and $\hat{T}$, agreeing with the statements of Lemma~\ref{lem:killT0} and Lemma~\ref{lem:dontkillThat} below.

In $T_{\sup}$, define a \emph{run} of entries to the entries in a single thread of cells, thought of as an increasing sequence. In this way, we define the $j$th run of $T_{\sup}$ to be the entries in the $j$th thread of $D(\alpha)$. In Example~\ref{ex:TsupThat}, the $4$th run of $T$ consists of the entries $7,8,9,10$. 

\begin{lemma}\label{lem:killT0}
The entries $x, x+1, \ldots , \varepsilon-2, \varepsilon-1$ all reside strictly left of $\varepsilon$ in $T_{\sup}$.
\end{lemma}
\begin{proof}

Since $\varepsilon$ is the rightmost entry in its run by Lemma~\ref{lem:rightmostinthread}, none of $x, x+1, \ldots , \varepsilon-2, \varepsilon-1$ belong to the run involving $\varepsilon$; they belong to strictly earlier runs. By definition, all entries of the run involving $x$ that are greater than $x$ are strictly left of $x$, and thus strictly left of $\varepsilon$. 

Consider any run using entries from $x, x+1, \ldots , \varepsilon-2, \varepsilon-1$ that is not the run containing $x$. Since such a run is strictly earlier than the run containing $\varepsilon$, it begins weakly to the left of where the run using $\varepsilon$ begins. This implies the starting entry of any run involving $x, x+1, \ldots , \varepsilon-2, \varepsilon-1$ is strictly to the left of $\varepsilon$ (the only way it could be in the same column as $\varepsilon$ is if $\varepsilon$ was the first and only entry of its run and was in the first column, but we know $\varepsilon$ is not in the first column).  We claim this run must in fact end at least two columns to the left of $\varepsilon$, which implies that all its entries are strictly left of $\varepsilon$. If such a run assigned an entry (say $y$) to a cell in the column of $x$ (i.e. immediately left of the column of $\varepsilon$), then since this run is later than the run involving $x$, it places $y$ strictly above $x$ (and thus strictly above $\varepsilon$) by Lemma~\ref{lem:threadedSW}. But then $y$ and $\varepsilon$ form two cells of a triple with $y<\varepsilon$, hence there must exist an entry $z$ right-adjacent to $y$ with $z<\varepsilon$. But this is impossible since entries increase upwards in columns of $T_{\sup}$. 
\end{proof}

\begin{lemma}\label{lem:dontkillThat}
In $\hat{T}$, all entries $x, x+1, \ldots , \varepsilon-2, \varepsilon-1$ reside strictly left of $\varepsilon$, and the entry left-adjacent to $\varepsilon$ is strictly smaller than $x$.
\end{lemma}
\begin{proof}
By definition, all entries $1, \ldots , \varepsilon-1$ occupy the same cell in $\hat{T}$ as they do in $T_{\sup}$, and by Lemma~\ref{lem:killT0}, all entries $x, x+1, \ldots \varepsilon-1, \varepsilon$ reside strictly left of $\varepsilon$ in $T_{\sup}$. By the proof of Lemma~\ref{lem:rightmostinthread}, $\varepsilon$ occupies a cell in $\hat{T}$ that is strictly right of the cell it occupies in $T_{\sup}$. So these entries must reside at least two columns to the left of $\varepsilon$ in $\hat{T}$, and thus none of them are left-adjacent to $\varepsilon$ in $\hat{T}$. Since the entry of the cell left-adjacent to $\varepsilon$ in $\hat{T}$ is strictly smaller than $\varepsilon$ (by (R1)), it must also be strictly smaller than $x$. 
\end{proof}

In Example~\ref{ex:TsupThat}, we have $\pi_{9}\pi_{10}\pi_{11}(T_{\sup}) = 0$ while $\pi_{9}\pi_{10}\pi_{11}(\hat{T}) = s_{9}s_{10}s_{11}(\hat{T})\neq 0$, agreeing with the statement of Corollary~\ref{cor:main} below.

\begin{corollary}\label{cor:main}
The operator $\pi_x \pi_{x+1} \ldots \pi_{\varepsilon-2} \pi_{\varepsilon-1}$ satisfies 
\begin{enumerate}
\item $\pi_x \pi_{x+1} \ldots \pi_{\varepsilon-2} \pi_{\varepsilon-1}(T_{\sup}) = 0$; and 
\item $\pi_x \pi_{x+1} \ldots \pi_{\varepsilon-2} \pi_{\varepsilon-1}(\hat{T}) = s_x s_{x+1} \ldots s_{\varepsilon-2} s_{\varepsilon-1}(\hat{T})\neq 0$. 
\end{enumerate}
\end{corollary}
\begin{proof}
For (1), by Lemma~\ref{lem:killT0} all of $x, x+1, \ldots , \varepsilon-2, \varepsilon-1$ are strictly left of $\varepsilon$ in $T_{\sup}$, and $x$ is left-adjacent to $\varepsilon$. Therefore, each operator $\pi_j$ for $x<j$ in the sequence exchanges the entry $j+1$ in the cell that contains $\varepsilon$ in $T_{\sup}$, with the entry $j$. Hence, after $\pi_{x+1}$ is applied, the entry right-adjacent to $x$ is $x+1$, and so applying $\pi_x$ yields $0$.

For (2), by Lemma~\ref{lem:dontkillThat}, all of $x, x+1, \ldots , \varepsilon-2, \varepsilon-1$ are strictly left of $\varepsilon$ in $\hat{T}$, and the entry left-adjacent to $\varepsilon$ is strictly smaller than $x$. Therefore, similarly to (1), each $\pi_j$ in the sequence of operators exchanges the entry $j+1$ in the cell that contains $\varepsilon$ in $\hat{T}$, with the entry $j$. Since the entry left-adjacent to $\varepsilon$ in $\hat{T}$ is strictly smaller than $x$, none of the operators $\pi_j$ in the sequence yield $0$, and in particular all of them act by $s_j$.
\end{proof}

Finally, recall the partial ordering on $\SYRT(\alpha)$ given in Lemma~\ref{lem:partialorder} and restrict this ordering to $E_0$. Define the \emph{rank} of $T\in E_0$ to be $p$ if there is a sequence of operators $\pi_{i_1} \ldots \pi_{i_p}$ satisfying $\pi_{i_1} \ldots \pi_{i_p}(T_{\sup}) = s_{i_1} \ldots s_{i_p}(T_{\sup}) = T$. Such a sequence must exist since $T_{\sup}$ is the source tableau of $E_0$, and it is straightforward to observe that $s_{i_1} \ldots s_{i_p}$ must be a reduced word in the symmetric group $S_n$. It follows that ${\rm rank}(T)$ is well-defined, and moreover that if $\pi_j(T) = s_j(T)$ for some $j$, then ${\rm rank}(\pi_j(T)) = {\rm rank}(T)+1$. We are now ready to prove Theorem~\ref{thm:main}. \\

\noindent
\emph{Proof of Theorem~\ref{thm:main}.} Recall that our goal is to show that if $f(T_{\sup}) = \sum_{T\in E_0} a_T T$, then $a_T=0$ for any $T\neq T_{\sup}$. Suppose for a contradiction that there exists some $T\in E_0$ that is not equal to $T_{\sup}$ and has nonzero coefficient. By Lemma~\ref{lem:descentofTprime}, $a_T=0$ whenever $T$ has a descent that is not a descent of $T_{\sup}$. Therefore, let $\hat{T}\neq T_{\sup}$ be of maximal rank such that $a_{\hat{T}} \neq 0$ and $\Des(\hat{T})\subseteq \Des(T_{\sup})$.

Let $\pi_x \pi_{x+1} \ldots \pi_{\varepsilon-2} \pi_{\varepsilon-1}$ be the sequence of operators from Corollary~\ref{cor:main}, i.e., $\pi_x \pi_{x+1} \ldots \pi_{\varepsilon-2} \pi_{\varepsilon-1}(T_{\sup}) = 0$ and $\pi_x \pi_{x+1} \ldots \pi_{\varepsilon-2} \pi_{\varepsilon-1}(\hat{T}) = s_x s_{x+1} \ldots s_{\varepsilon-2} s_{\varepsilon-1}(\hat{T}) = T'\neq 0$.   
We claim that if $a_T\neq 0$ for some $T\in E_0$ and  $\pi_x \pi_{x+1} \ldots \pi_{\varepsilon-2} \pi_{\varepsilon-1}(T) = T'$, then in fact $T=\hat{T}$. To see this, note that in order to be equal, $\pi_x \pi_{x+1} \ldots \pi_{\varepsilon-2} \pi_{\varepsilon-1}(T)$ and $\hat{T}$ must have the same rank. However, by assumption ${\rm rank}(\hat{T})\ge {\rm rank}(T)$, and each $\pi_i$ acts by $s_i$ when applied in sequence to $\hat{T}$. Since each application of $s_i$ raises rank by one, the only way these two tableaux can have the same rank is for $T$ and $\hat{T}$ to have the same rank and for each $\pi_i$ to also act by $s_i$ when applied in sequence to $T$. But then we have $s_x s_{x+1} \ldots s_{\varepsilon-2} s_{\varepsilon-1}(T) = s_x s_{x+1} \ldots s_{\varepsilon-2} s_{\varepsilon-1}(\hat{T})$, and it follows that $T=\hat{T}$ since each $s_i$ is injective.

Therefore, the coefficient of $T' = \pi_x \pi_{x+1} \ldots \pi_{\varepsilon-2} \pi_{\varepsilon-1}(\hat{T})$ in 
\[\pi_x \pi_{x+1} \ldots \pi_{\varepsilon-2} \pi_{\varepsilon-1}(f(T_{\sup})) = \sum_{S\in E_0} a_S\pi_x \pi_{x+1} \ldots \pi_{\varepsilon-2} \pi_{\varepsilon-1}(S)\]
is precisely $a_{\hat{T}}$. On the other hand, 
\[\pi_x \pi_{x+1} \ldots \pi_{\varepsilon-2} \pi_{\varepsilon-1}(f(T_{\sup})) = f(\pi_x \pi_{x+1} \ldots \pi_{\varepsilon-2} \pi_{\varepsilon-1}(T_{\sup})) = f(0) = 0.\]
Hence $a_{\hat{T}}=0$, contradicting our assumption $a_{\hat{T}}\neq 0$.

Therefore $f(T_{\sup}) = a_{T_{\sup}}T_{\sup}$, and ${\bf R}_\alpha^{E_0}$ is indecomposable. It follows that ${\bf R}_\alpha$ is indecomposable if ${\bf R}_\alpha={\bf R}_\alpha^{E_0}$, which is exactly the case when $\alpha$ is simple. Since we have already observed that ${\bf R}_{\alpha}$ is decomposable when $\alpha$ is not simple, this completes the proof of Theorem~\ref{thm:main}.

%
%

\bibliographystyle{amsalpha} 
\bibliography{RowStrictModule}

\end{document}